\documentclass[11pt]{article}

\usepackage{amsmath,amsthm,verbatim,amssymb,amsfonts,amscd, graphicx}
\usepackage{graphics}
\usepackage{tikz-cd}
\usepackage{subcaption}
\theoremstyle{plain}
\newtheorem{theorem}{Theorem}

\newtheorem{lemma}{Lemma}
\newtheorem{remark}{Remark}
\newtheorem{proposition}{Proposition}

\theoremstyle{definition}

\begin{document}

\title{Kernels of splitting homomorphisms}
\author{Michael R. Klug}
\maketitle

\begin{abstract}
   Lei and Wu have given a description of the second homotopy group of a closed orientable 3-manifold in terms of the kernels of the epimorphisms from the fundamental group of a Heegaard splitting surface onto the fundamental groups of the two handlebody sides.  In this note, we give a geometric derivation of this result and collect some observations about the relation between the various groups and the topology of the 3-manifold and the Heegaard splitting.   
\end{abstract} 

\section{Introduction} \label{sec:intro}

The role of the fundamental group in 3-dimensional topology is central.  In \cite{stallings66hownot}, Stallings put forth an approach to the 3-dimensional Poincar\'e conjecture by way of the group theory associated to a Heegaard splitting of a homotopy 3-sphere.  This approach was furthered by Jaco \cite{jaco1969splitting} who proved that several group-theoretic statements involving free groups and surface groups were equivalent to the Poincar\'e conjecture.  

Let $M$ be a compact orientable 3-manifold and let $M = H_\alpha \cup H_\beta$ be a Heegaard splitting of $M$ with $\Sigma = H_\alpha \cap H_\beta$ an orientable surface of genus $g$.
Fix a point $\ast$ in $\Sigma$ and consider the associated pushout diagram of groups, with morphisms induced by inclusions.  
\[
\begin{tikzcd}	
	\pi_1(\Sigma, \ast) \ar[r, two heads] \ar[d, two heads] &
	\pi_1(H_\alpha,\ast) \ar[d, two heads] \\
	\pi_1(H_\beta, \ast) \ar[r, two heads] & \pi_1(M,\ast) 
\end{tikzcd}
\]

The maps from the surface group to the free groups
of the handlebodies are surjective as any curve in a handlebody can be made to miss the spine of the handlebody and thereby can be homotoped to the boundary.  The map $\phi : \pi_1(\Sigma) \to \pi_1(H_\alpha) \times \pi_1(H_\beta)$ is the \emph{splitting homomophism} associated to the Heeagard surface $\Sigma$.  Jaco showed how all of the topology of $M$ is encoded in the splitting homomorphism \cite{jaco1969splitting}.   From here out, we will not generally mention the basepoints explicitly.  

Building on the approach of Stallings and Jaco, Hempel \cite{hempel3manifolds} showed the condition in the following theorem is equivalent to the Poincar\'e conjecture and therefore, by the work of Perelman \cite{perelman1}, \cite{perelman2}, it follows:

\begin{theorem}(Perelman)
    
For all integers $g \geq 0$ and any pair of surjective homomorphisms $\phi_1, \phi_2:  \pi_1(\Sigma) \to F_g \times F_g$ where $\Sigma$ is a genus $g$ closed orientable surface, and $F_g$ is a rank $g$ free group, there is an isomorphism 

\[
\begin{tikzcd}
\pi_1(\Sigma) \ar[r, "\cong"] & \pi_1(\Sigma) 
\end{tikzcd}
\]
and a pair of isomorphisms
\[
\begin{tikzcd}
F_g \ar[r, "\cong"] & F_g 
\end{tikzcd}
\]
such that the following commutes
\[
	 \begin{tikzcd}[column sep=large]
		 \pi_1(\Sigma) \ar[r, "\phi_1", two heads] \ar[d, "\cong"]       & F_g \times F_g  \ar[d,"\cong", xshift=2.0ex] \ar[d,"\cong", xshift=-2.0ex] \\ 
		 \pi_1(\Sigma) \ar[r, "\phi_2", two heads]            & F_g \times F_g
	 \end{tikzcd}
\]

In other words, there is a unique surjective homomorphism $\pi_1(\Sigma) \to F_g \times F_g$ up to pre- and post-composition with automorphisms (where post-composition is by automorphisms that are a product of automorphisms on the $F_g$ factors).  
\end{theorem}

More recently, a similar group-theoretic statement equivalent to the smooth 4-dimensional Poincar\'e conjecture has been given by Abrams, Gay, and Kirby \cite{abrams2018group}.  In considering these connections, we wanted to understand how the basic topological invariants of a space can be seen from the perspective of splitting homomorphisms.  Let $K_\alpha = \ker(\pi_1(\Sigma) \to \pi_1(H_\alpha))$ and $K_\beta = \ker(\pi_1(\Sigma) \to \pi_1(H_\beta))$.  If the handlebodies $H_\alpha$ and $H_\beta$ are described using a Heegaard diagram, then $K_\alpha$ and $K_\beta$ are normally generated by the curves describing $H_\alpha$ and $H_\beta$, respectively.    In this note, we give a geometric proof of a result of Lei and Wu \cite{lei2011kernels} that shows how to compute $\pi_2(M)$ in terms of $K_\alpha$ and $K_\beta$ (see Theorem \ref{thm:pi_2}).   

\section{Preliminary lemmas and general remarks}

We say a surface $F$ is of finite type if $\pi_1(F)$ is finitely generated, otherwise we say that $F$ has infinite type.  Surfaces of finite type are up to homeomorphism determined by their genus, number of boundary components, and number of punctures.  Surfaces of infinite type also admit a classification in terms of their genus, number of boundary components, and space of ends \cite{kerekjarto}, \cite{richards}.  

For a surface of finite type $F$ with genus greater than or equal to 2, together with a choice of hyperbolic metric on $F$, there are at most finitely many many closed geodesics of length less than a given constant $L \in \mathbb{R}$ \cite{buser2010geometry}.  

\begin{theorem}\label{thm:1}
        If $g \geq 2$, then $K_\alpha \cap K_\beta$ is a not-finitely-generated free group.  
\end{theorem}

\begin{proof}
    In \cite{jaco1970certainsubgroups}, it is shown that $K_\alpha \cap K_\beta \neq 1$.  Let $\widetilde{\Sigma}$ be the cover of $\Sigma$ corresponding to $K_\alpha \cap K_\beta$.  Then $\widetilde{\Sigma}$ is noncompact since $K_\alpha \cap K_\beta$ has infinite index in $\pi_1(\Sigma)$.  Therefore, $K_\alpha \cap K_\beta$ is free since noncompact surfaces have free fundamental groups \cite{johansson}. Since $g \geq 2$, the surface $\widetilde{\Sigma}$ obtains a hyperbolic metric by pulling back a hyperbolic metric on $\Sigma$.   Since $K_\alpha \cap K_\beta$ is normal, the deck translations of the covering act on $\widetilde{\Sigma}$ as isometries.  Let $\gamma$ be a closed geodesic in $\widetilde{\Sigma}$.  Then all of the infinite translates of $\gamma$ have the same length as $\gamma$.  Thus, $\widetilde{\Sigma}$ is of infinite type, by the preceding discussion, and therefore $K_\alpha \cap K_\beta$ is not finitely generated.  
\end{proof}

Recall that a 3-manifold $M$ is called \textit{reducible} if there is an embedded sphere in $M$ that does not bound a 3-ball in $M$, and \textit{irreducible} otherwise.
Equivalently, by the Sphere Theorem \cite{papakyriakopoulos1957dehn},
$M$ is irreducible if and only if $\pi_2(M) = 0$. 
A Heegaard splitting  $M = H_\alpha \cup_\Sigma H_\beta$ is called
\textit{reducible} if there is an essential simple closed curve in $\Sigma$
that bounds embedded disks in both $H_\alpha$ and $H_\beta$;
and \textit{irreducible} otherwise.  Haken's lemma \cite{Haken}
asserts that any Heegaard splitting of a reducible 3-manifold is reducible.

\begin{proposition}\label{no embedded}
The  Heegaard splitting $M = H_\alpha \cup_\Sigma H_\beta$  is reducible if and only if the subgroup $K_\alpha \cap K_\beta$ contains a nontrivial element that can be represented by an embedded curve.  If $M$ is reducible, then $K_\alpha \cap K_\beta$ contains a nontrivial element that can be represented by an embedded curve.   
\end{proposition}

\begin{proof}
    The first follows immediately from the definitions.  The second part follows from Haken's lemma.  
\end{proof}

Note that given any irreducible Heegaard splitting of an 3-manifold,
the group $K_\alpha \cap K_\beta$ is a nontrivial subgroup of a surface group that,
by Proposition \ref{no embedded}, cannot contain any elements
that can be represented by embedded curves.
Other examples of this phenomenon, in fact finite index examples,
are known (see \cite{malestein2010selfintersection}, \cite{livingston2000maps}).  

The following lemma is immediate upon considering the
``$4g$-gon with a hole'' picture of a genus $g$ surface 
with one boundary component.      

\begin{lemma}
    \label{lem:commutator}
    Let $X$ be a topological space with basepoint $p \in X$ and let
    $A, B \leq \pi_1(X,p)$ be two normal subgroups.
    An element $\gamma \in \pi_1(X,p)$ is in $[A,B]$ if an only if
    there is a continuous map
    $f \colon (S,\partial S) \to (X,\gamma)$ where $S$ is a genus $g$ orientable surface
    with one boundary component such that the images of the curves
    $a_i$ and $b_j$ in Figure \ref{fig:surface_grope}
    are in the normal subgroups
    $A$ and $B$ respectively for $1 \leq i,j \leq g$.  
\end{lemma}

\begin{figure}
    \centering
    \includegraphics[scale=0.4]{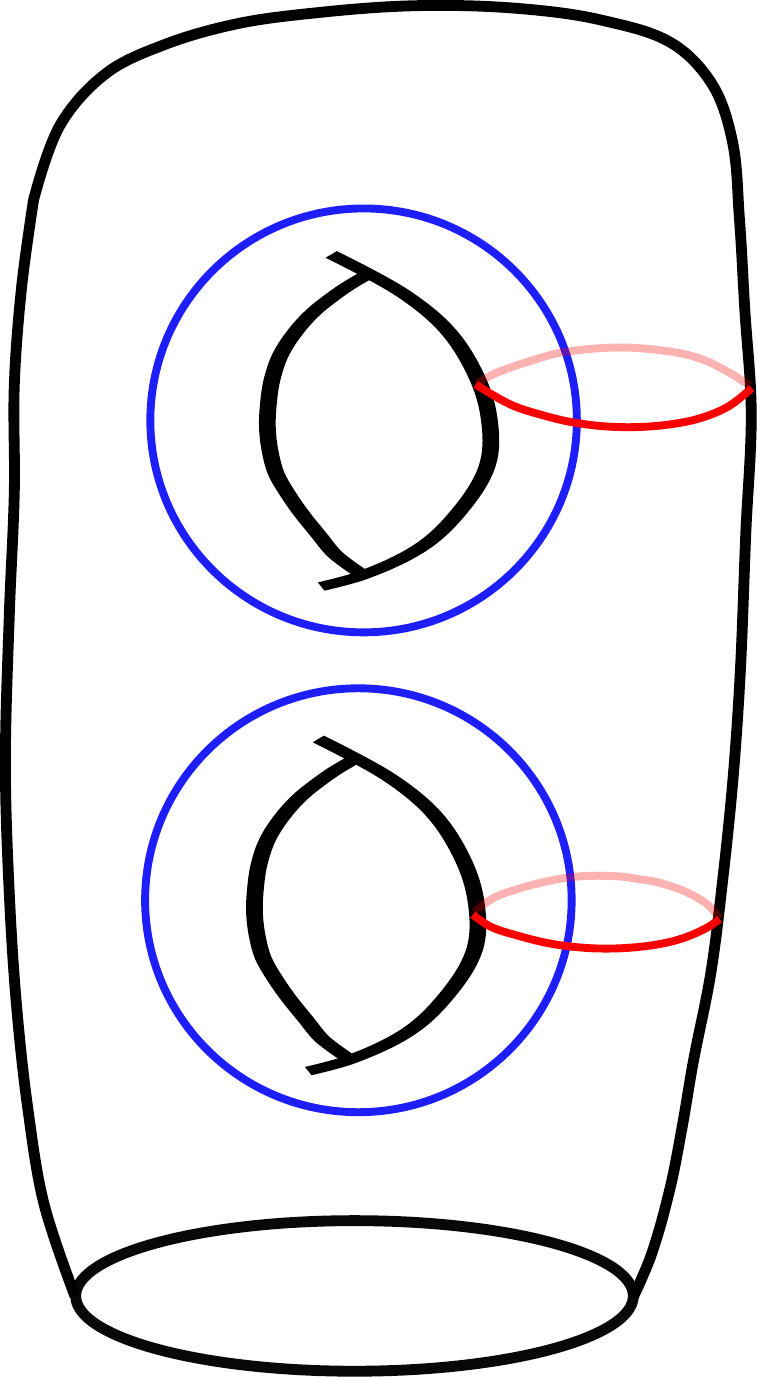}
    \centering
    \caption{This figure shows a genus 2 example.  The basepoint is assumed to be on the boundary of the surface.  The curves in red are the $a_i$ curves and the curves in blue are the $b_j$ curves.  Note here that $A$ and $B$ are assumed to be normal subgroups, so the lack of a basepoint for the curves is irrelevant -- i.e., if some choice of arcs to the basepoint results in based curves that are in $A$ and $B$, respectively, then so too does any other choice of arcs.  }
    \label{fig:surface_grope}
\end{figure}

Let $F$ be a surface in a 3-manifold $M$.
By \textit{adding a tube} to $F$, we mean creating a new surface $F'$,
that is obtained from $F$ by taking the symmetric difference of $F$
with the boundary of an embedded $(D^2 \times [0,1], D^2 \times \{ 0,1\}) \to (M, F)$
such that the image of $D^2 \times (0,1)$ is disjoint from $F$.

\begin{lemma}
    \label{lem:tubes}
    Let $F$ be a compact surface in $S^3$.
    There is a surface $F'$ obtained from adding tubes to $F$
    such that $F'$ is isotopic to the standard genus $g$ Heegaard splitting surface of $S^3$.  Moreover, if $F$ is a surface in $B^3$ with $F \cap \partial B^3$ connected, then we can add tubes to $F$ to obtain a standard surface as in Figure
    \ref{fig:trivial_in_ball}.  
\end{lemma}

\begin{proof}
    Recall that any 3-manifold $M$ with nonempty boundary has a handle decomposition with a 0-handle, 1-handles, and then 2-handles -- one way of seeing this is that $M$ admits a Morse function that is constant on the boundary and that has increasing index critical points.
    By carving out the 2-handles, we obtain a handlebody and therefore, from every 3-manifold with boundary, we can obtain a handlebody if we carve out enough tubes.  Now let $F$ be a surface in $S^3$.
    Note that carving tubes out of one side of $F$ corresponds to adding 1-handles to the other side, and notice that adding 1-handles to a handlebody produces another handlebody (with higher genus).  Therefore, by adding tubes to one side of $F$, we can obtain a new surface that bounds a handlebody on one side, and by adding tubes to the other side, we can obtain a surface that bounds handlebodies on both sides.   The moreover statement is obtained from this result by adding a ball with a standard equator disk to the $B^3$ containing $F$, with the boundary of the disk glued to the boundary of $F$, and performing the above argument with all of the tubes not intersecting this trivial disk-ball pair.  
\end{proof}

\begin{figure}
    \centering
    \includegraphics[scale=0.4]{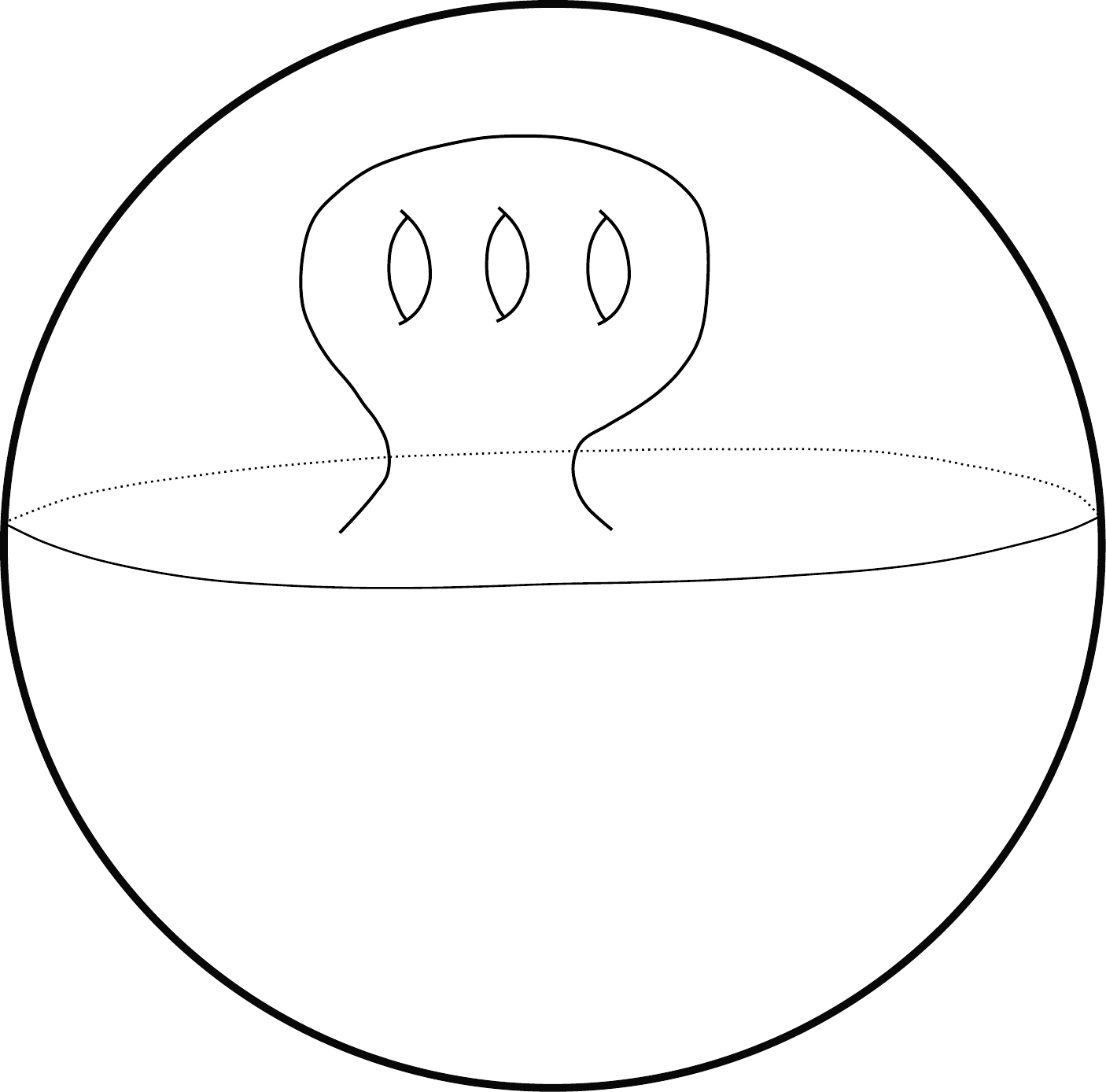}
    \centering
    \caption{This figure shows a genus 3 example of the trivial surface in a 3-ball.  The boundary of the surface is the equator and the surface is the result of stabilizing (as with Heegaard splittings) the equitorial disk three times.   }
    \label{fig:trivial_in_ball}
\end{figure}

\begin{lemma}
    \label{lem:surjective}
    Let $\Sigma$ be a Heegaard splitting surface of a closed orientable 3-manifold $M$,
    $F$ be a compact connected surface, and $f \colon F \to M$ be a continuous map.
    Then $f$ can be homotoped so that $f^{-1}(\Sigma)$ is connected.  
\end{lemma}

\begin{proof}
    First, homotope $f$ such that it is an immersion transverse to $\Sigma$.  Let $\gamma_1$ and $\gamma_2$ be two distinct connected components of $f^{-1}(\Sigma)$ such that there is an arc $a \subset F$ with endpoints in $\gamma_1$ and $\gamma_2$ respectively and such that $f(a)$ is entirely contained in $H_\alpha$ or $H_\beta$.  Note that any properly embedded arc in a handlebody can be homotoped relative to its endpoints so as to be contained in the boundary of the handlebody.
    Therefore, by homotoping $f$ close to $a$ following this homotopy of $a$ into $\Sigma$ we obtain a new map with one fewer connected component in the inverse image of $\Sigma$.  Repeating this process then gives the result.   
\end{proof}

\section{$\pi_2(M)$ from splitting homomorphisms}

The following result is given by Lei and Wu in \cite{lei2011kernels}, where
it is stated that it follows from general methods in \cite{brown1987vanKampen}.  Here we provide a hands-on geometric argument.

\begin{theorem} (Lei and Wu)
    \label{thm:pi_2}
    For every closed $3$-manifold $M$,
    there is an isomorphism of abelian groups
    \begin{equation*}
        \pi_2(M)
        \cong
        (K_\alpha \cap K_\beta) / [K_\alpha, K_\beta]
    \end{equation*}
    
    Let $p \colon \pi_1(\Sigma) \twoheadrightarrow \pi_1(M)$
    be the surjection induced by inclusion.
    Then $(K_\alpha \cap K_\beta) / [K_\alpha, K_\beta]$ has a well-defined
    $\pi_1(M)$-action given by $g \cdot [\gamma] = [g' \gamma g'^{-1}]$
    for $g \in \pi_1(M)$ with $g'$ a choice of preimage $p(g') = g$ and
    $[\gamma] \in (K_\alpha \cap K_\beta) / [K_\alpha, K_\beta]$.
    With respect to this action, the above isomorphism is a
    $\mathbb{Z} \pi_1(M)$-module isomorphism.  
\end{theorem}

\begin{proof}
    We first define a map
    $\phi \colon K_\alpha \cap K_\beta \to \pi_2(M)$.
    Notice that, given a curve $\gamma \in K_\alpha \cap K_\beta$ then
    by the definition of $K_\alpha$ and $K_\beta$,
    there exist disks $D_\alpha \subset H_\alpha$ and $D_\beta \subset H_\beta$
    such that $\gamma = \partial D_\alpha = \partial D_\beta$.
    If $D'_\alpha$ and $D'_\beta$ are other such disks,
    then the spheres $D_\alpha \cup_\gamma D_\beta$ and $D'_\alpha \cup_\gamma D'_\beta$
    are homotopic with the homotopy fixing the basepoint,
    since $\pi_2$ of a handlebody is trivial
    and so the union of the disks can be extended to a
    map from a 3-ball.
    Therefore, we have a map $\phi$ 
    which is seen to be a group homomorphism.
    
    The surjectivity of $\phi$ is the content of Lemma \ref{lem:surjective} in the case where $F$ is a sphere.  
    
    We will now show that the map $\phi$ descends
    to the quotient $K_\alpha \cap K_\beta / [K_\alpha, K_\beta]$.
    As a preliminary remark, observe that the commutator
    $[K_{\alpha}, K_{\beta}]$ is a subgroup of the intersection
    $K_{\alpha} \cap K_{\beta}$,
    because the kernels $K_{\alpha}, K_{\beta}$
    are normal subgroups of $\pi_{1}(\Sigma)$.
    To see that $[K_\alpha, K_\beta] \leq \ker(\phi)$
    note that by Lemma \ref{lem:commutator}, there exists a surface $S$
    as in Figure \ref{fig:surface_grope} and a continuous map
    $f \colon (S, \partial S) \to (\Sigma, \gamma)$ such that
    $f(a_i) \in K_\alpha$ and $f(b_j) \in K_\beta$.
    Therefore, we have immersed disks $D^\alpha_i, D^\beta_i $ with
    $\partial D^\alpha_i = f(a_i)$ and $\partial D^\beta_j = f(b_j)$.
    We thus have a capped surface (i.e., the result of adding to the $a_i$ and $b_j$ curves in Figure \ref{fig:surface_grope})
    mapping into $M$ with the boundary mapping to $\gamma$ -- call the image $\hat{S}$.
    Now consider a neighborhood of $\hat{S}$ which is topologically a ball.  Thus, we obtain a map of a 2-sphere  whose equator maps to $\gamma$ and such that the two disks bounding the equator map to disks $D_\alpha$ and $D_\beta$ in $H_\alpha$ and $H_\beta$, respectively.
    Then $\phi(\gamma) = 0$ since $D_\alpha \cup_\gamma D_\beta$ bounds a ball.  
    
    We now prove injectivity of the resulting map
    $$
    \phi: (K_\alpha \cap K_\beta) / [K_\alpha, K_\beta] \to \pi_2(M)
    $$
    Suppose that $\gamma \in \ker(\phi)$ and let
    $g \colon B^3 \to M$ with
    $g(\partial B^3) = D_\alpha \cup_\gamma D_\beta = \phi(\gamma)$.
    If $g^{-1}(\Sigma)$ is standard as in Figure 
    \ref{fig:trivial_in_ball}, then, 
    by Lemma \ref{lem:commutator}, we have that
    $\gamma \in [K_\alpha, K_\beta]$.
    If $g^{-1}(\Sigma)$ is not standard, then,
    by homotoping $g$, we can add tubes
    to $g^{-1}(\Sigma)$, since such a tube maps to a thickened arc in $M$ that is either contained in $H_\alpha$ or $H_\beta$ (where as in the proof of Lemma \ref{lem:tubes}, every arc can be homotoped to lie in the boundary).  Therefore, by Lemma \ref{lem:tubes},
    we can make $g^{-1}(\Sigma)$ standard.
    Thus $\ker(\phi) \leq [K_\alpha, K_\beta]$
    and therefore $\phi$ induces an isomorphism of abelian groups
    $K_\alpha \cap K_\beta / [K_\alpha, K_\beta] \cong \pi_2(M)$.

    Now, consider the action of $\pi_1(M)$ on $(K_\alpha \cap K_\beta)/[K_\alpha, K_\beta]$ as in the statement of the result.  By our construction of $\phi$ and the definition of the action of $\pi_1(M)$ on $\pi_2(M)$, we see that $\phi$ is in fact an isomorphism of $\mathbb{Z}[\pi_1(M)]$-modules, thus proving the result.  
\end{proof}

\begin{remark}	
    We have not been discussing the subgroup
    $K_\alpha K_\beta \leq \pi_1(\Sigma)$, which is another subgroup of interest (here $K_\alpha K_\beta$ denoted the join of $K_\alpha$ and $K_\beta$ -- i.e., the smallest subgroup containing boht of them).  
    Note that $\pi_1(\Sigma) / K_\alpha K_\beta \cong \pi_1(M)$
    (see \cite{stallings66hownot}) and the cover of $\Sigma$
    corresponding to $K_\alpha K_\beta$ is the preimage of $\Sigma$
    in the universal cover for $M$.
    It then follows, as in Theorem \ref{thm:1}, that either 
	\begin{enumerate}
	\item $\pi_1(M)$ is finite and $K_\alpha K_\beta$ is isomorphic to the fundamental group of a closed orientable surface, or
	\item $\pi_1(M)$ is infinite and $K_\alpha K_\beta$ is an infinitely generated free group.  
	\end{enumerate}
\end{remark}

\bibliographystyle{plain} 
\bibliography{references}

\end{document}